\documentclass{article}

\usepackage{amsmath}
\usepackage{amsfonts}
\usepackage{amssymb}
\usepackage{amsthm}
\usepackage{microtype}
\usepackage[sf,bf,uppercase,tiny]{titlesec}
\usepackage{graphicx}

\newtheorem{theorem}{Theorem}[section]
\newtheorem{lemma}[theorem]{Lemma}
\newtheorem{corollary}[theorem]{Corollary}
\newtheorem{main}{Theorem}

\newcommand{\centerof}[1]{\mathbf{Z}(#1)}
\newcommand{\centralizer}[2]{\mathbf{C}_{#1}(#2)}

\newcommand{\id}{\{1\}}

\begin{document}

{\large \begin{center}
\textbf{\textsf{THE CYCLIC GRAPH (DELETED ENHANCED POWER GRAPH) OF A DIRECT PRODUCT}}
\end{center}}

\begin{center}
\textbf{\textsf{DAVID G. COSTANZO, MARK L. LEWIS, STEFANO SCHMIDT, EYOB TSEGAYE, GABE UDELL}}
\end{center}

\begin{abstract}
Let $G$ be a finite group.
Define a graph on the set ${G^{\#}=G\setminus\id}$ by declaring distinct elements $x,y\in G^{\#}$ to be adjacent if and only if $\langle x,y\rangle$ is cyclic.
Denote this graph by $\Delta(G)$.
The graph $\Delta(G)$ has appeared in the literature under the names \textit{cyclic graph} and \textit{deleted enhanced power graph}. 
If $G$ and $H$ are nontrivial groups, then $\Delta(G\times H)$ is completely characterized.
In particular, if $\Delta(G\times H)$ is connected, then a diameter bound is obtained, along with an example meeting this bound.
Also, necessary and sufficient conditions for the disconnectedness of $\Delta(G\times H)$ are established.
\end{abstract}

\section{Introduction}

All groups in this note are finite.
Our definition of a \textit{graph} shall be an ordered pair $(V,E)$, where $V$ is a finite set and $E$ is a set of unordered pairs of elements in $V$.
The elements in $V$ are called \textit{vertices}; the elements in $E$ are called \textit{edges}.
If $\{x,y\}\in E$, then it is customary to write $x\sim y$ and to say that the points $x$ and $y$ are \textit{adjacent}.
Observe that $\sim$ is an irreflexive, symmetric relation.
Graphs are often thought of diagrammatically; the vertex set is represented as an array of points and an edge is a line segment (or arc) connecting distinct points.
If $\Gamma$ is a graph, then $c(\Gamma)$ denotes the number of connected components of $\Gamma$.
 
Let $G$ be a group.
Define a graph $\Delta(G)$ on $G$ as follows.
Let $G^{\#}=G\setminus\id$ be the vertex set, and let $\{x,y\}$ ($x$ and $y$ distinct elements in $G^{\#}$) be an edge if and only if $\langle x,y\rangle$ is cyclic.
The graph $\Delta(G)$ is called the \textit{cyclic graph} of $G$.
If, for distinct elements $x,y\in G^{\#}$, the subgroup $\langle x,y\rangle$ is cyclic, then write $x\sim y$.

The cyclic graph of a group has appeared in the literature previously.
In \cite{imp}, Imperatore considered both finite and infinite groups satisfying the condition that the relation $\sim$ is transitive.
(Notice that this condition is equivalent to assuming that each connected component of the cyclic graph is complete.)
More specifically, Imperatore combined the transitivity of the relation $\sim$ with group-theoretic conditions such as abelian, nilpotent, and supersolvable to obtain strong classification results.

Imperatore and the second author continued to study the structural implications of the assumption that the relation $\sim$ is transitive in \cite{implewis}.
In particular, they obtained a complete classification of \textit{finite} groups satisfying this assumption.

The cyclic graph has also appeared in the literature under the name \textit{deleted enhanced power graph}.
Bera and Bhuniya (\cite{bera}, Section 5) study the deleted enhanced power graph of a group $G$, and a few of the lemmas below recover similar results.
In \cite{cameron}, the \textit{enhanced power graph}, which is our cyclic graph with the identity element \textit{included} as a vertex, is studied as a sort of layer between the power graph and the commuting graph of a group.

Let $G$ and $H$ be nontrivial groups.
The primary focus in this note is the graph $\Delta(G\times H)$.
It turns out that merely assuming the connectedness of the cyclic graph of a direct product leads to a diameter bound.

\begin{main}
If $G$ and $H$ are nontrivial groups and the graph $\Delta(G\times H)$ is connected, then $\textnormal{diam}(\Delta(G\times H))\le 7$.
\end{main}

The direct product $G\times H$, where $G$ is SmallGroup(1944,2320) in the GAP Small Groups library and $H$ is the Frobenius group $(C_{9}\times C_{9})\rtimes C_{4}$, constitutes an example where $\textnormal{diam}(\Delta(G\times H))= 7$.

Let $G$ and $H$ be groups where $\Delta(G)$ and $\Delta(H)$ are disconnected.
The graph $\Delta(G\times H)$ may be connected, possibly suggesting that there exists no relationship between the connectedness of $\Delta(G\times H)$ and that of $\Delta(G)$ or $\Delta(H)$. 
There are scenarios, however, where the connectedness of $\Delta(G\times H)$ does influence $\Delta(G)$ or $\Delta(H)$.

\begin{main}
If $G$ and $H$ are groups and the graph $\Delta(G\times H)$ is connected with $\textnormal{diam}(\Delta(G\times H))\le 2$, then $\Delta(G)$ and $\Delta(H)$ are connected with $$\textnormal{diam}(\Delta(G))\le 2\qquad\hbox{or}\qquad\textnormal{diam}(\Delta(H))\le 2.$$
\end{main}

Let $G$ be a group, and let $p$ be a prime.
If there exists an element in $G$ of order $p$ with a $p$-group centralizer, then $G$ is said to satisfy $(\mathcal{C}_{1}(p))$.
This centralizer condition yields a necessary and sufficient condition for the cyclic graph of a direct product of nontrivial groups to be disconnected.

\begin{main}
Let $G$ and $H$ be nontrivial groups.
The graph $\Delta(G\times H)$ is disconnected if and only if there exists a prime $p$ such that $G$ and $H$ satisfy $(\mathcal{C}_{1}(p))$.
\end{main}

Let $G$ and $H$ be nontrivial groups, and let $p$ be a prime.
Note that $G\times H$ satisfies $(\mathcal{C}_{1}(p))$ if and only if $G$ and $H$ satisfy $(\mathcal{C}_{1}(p))$.
So, the previous theorem can also be framed in terms of a centralizer condition on $G\times H$.

This research was conducted during a summer REU at Kent State University with the funding of NSF Grant DMS-1653002.
We thank the NSF for their support.
The first, third, fourth, and fifth authors also thank the faculty and staff at Kent State University for their hospitality.

\section{Preliminaries}


There are many graphs that can be assigned to a group.
A few graphs most relevant to the cyclic graph are now mentioned.

Let $G$ be a group.
The \textit{punctured power graph} of $G$ is the graph that has $G^{\#}$ as its vertex set and has an edge drawn between distinct vertices $x$ and $y$ if and only if $x\in\langle y\rangle$ or $y\in\langle x\rangle$.
Note that the punctured power graph of $G$ is a subgraph of $\Delta(G)$, and so the connectedness of the punctured power graph implies connectedness of $\Delta(G)$.
It is easy to check that the converse holds as well.
If the identity element of $G$ is included as a vertex in the punctured power graph, then this graph is simply called the \textit{power graph} of $G$.
Both the power graph and punctured power graph of a group are considered in \cite{shi}.
Jafari \cite{jafari} also considers the punctured power graph of a group and obtains similar results to the previously mentioned papers.
Several of the following lemmas proved in this section are motivated by results regarding the punctured power graph and yield results for the cyclic graph that are analogous to results proved in \cite{curtin} and \cite{doostabi}.

Let $G$ be a group.
The \textit{commuting graph} of $G$ is the graph that has $G\setminus\centerof{G}$ as its vertex set and has an edge drawn between distinct vertices $x$ and $y$ if and only if $xy=yx$.
We note that there are a number of different formulations for the definition of the commuting graph.
For example, in \cite{aschbacher}, the commuting graph is defined as the graph whose vertex set is the set of subgroups of order $p$ and whose edges are commuting subgroups, where $p$ is a fixed prime.
The commuting graph of a finite solvable group with trivial center was studied in \cite{parker}.
The papers \cite{giudici} and \cite{morgan} are also concerned with the commuting graph of a group. 

Much of our subsequent work depends on the following observation.
Note that this result appears in many graduate algebra and group theory texts, see, for example, Exercise $6$ in \cite{rose}.

\begin{lemma}\label{coprimeandcommute}
Let $G$ be a group, and let $x,y\in G$.
If $x$ and $y$ are commuting elements with coprime orders, then the subgroup $\langle x, y \rangle$ is cyclic.
\end{lemma}

Recall that a group $G$ is \textit{nilpotent} if it is the (internal) direct product of its Sylow subgroups.
The structure of $\Delta(G)$ for a nilpotent group $G$ is easily understood.

We need an easy lemma.

\begin{lemma}\label{orderpelements}
If $G$ is a $p$-group for some prime $p$, $x\in G$ with $o(x)=p$, and $\Upsilon$ is the connected component of $\Delta(G)$ containing $x$, then $x\sim y$ for every $y\in\Upsilon$.
\end{lemma}

\begin{proof}
Aiming for a contradiction, suppose that $x$ is \textit{not} adjacent to every point in $\Upsilon$.
Then there exist points $u,y\in\Upsilon$ such that $x\sim u\sim y$ and $x$ is not adjacent to $y$.
Since $x\sim u$, the subgroup $\langle x,u\rangle$ has a unique subgroup of order $p$: namely $\langle x\rangle$.
Note that $\langle u\rangle$ also has a unique subgroup of order $p$.
It follows that $\langle x\rangle\le\langle u\rangle$ as $\langle u\rangle\le\langle x,u\rangle$.
Now $x,y\in\langle u,y\rangle$.
But $\langle u,y\rangle$ is cyclic, and therefore so is $\langle x,y\rangle$, a contradiction.
\end{proof}

The following lemma concerns $\Delta(G)$ for a group $G$ of prime power order.
Theorem 2.5 in \cite{doostabi} establishes the analogous and, in fact, equivalent result for the punctured power graph.
We include the short proof here as it motivates the following work.

\begin{lemma}\label{subgroupsoforderp}
If $G$ is a $p$-group for some prime $p$, then $c(\Delta(G))$ equals the number of subgroups of $G$ with order $p$.
\end{lemma}

\begin{proof}
Let $\Upsilon$ be a connected component of $\Delta(G)$.
The set $\Upsilon$ contains an element of order $p$, say $x$.
Suppose that $y\in\Upsilon$ also has order $p$.
By Lemma \ref{orderpelements}, $x\sim y$.
Note that $\langle x\rangle=\langle y\rangle$ since $\langle x,y\rangle$ has a unique subgroup of order $p$.
Hence $\Upsilon\mapsto\langle x\rangle$ is a well-defined mapping from the connected components of $\Delta(G)$ to the subgroups of $G$ with order $p$.
The bijectivity of this function is clear. 
\end{proof}

Nilpotent groups that are not groups of prime power order are now considered.
Recall that elements with coprime orders commute in nilpotent groups.
(See, for example, Theorem 2.12 in \cite{suzukivolii}.)
Comparing the following lemma with Theorem 2.6 in \cite{doostabi}, we see that the upper bound on the diameter of the cyclic graph of a nilpotent group whose order is divisible by at least two primes is smaller than the diameter of its punctured power graph.

\begin{lemma}\label{nilpotentatleasttwoprimes}
If $G$ is nilpotent and $|G|$ is divisible by at least two primes, then $\Delta(G)$ is connected with $\textnormal{diam}(\Delta(G))\le 3$.
\end{lemma}

\begin{proof}
Let $x,y\in G^{\#}$.
Suppose that there exists a prime $p$ that divides $o(x)$ and does not divide $o(y)$.
Note that $o(x^{t})=p$ for suitable integer $t$, and so $x\sim x^{t}\sim y$.

Assume that every prime divisor of $o(x)$ is a prime divisor of $o(y)$.
Also, assume that every prime divisor of $o(y)$ is a prime divisor of $o(x)$.
(Otherwise, apply the argument in the previous paragraph.)
If $o(x)$ and $o(y)$ are powers of some prime $p$, then, using our hypothesis, fix $a\in G^{\#}$ with $o(a)=q$, where $q\neq p$ is a prime.
Observe that $x\sim a\sim y$.
If $o(x)$ and $o(y)$ are not powers of some prime, then there exist suitable integers $r$ and $s$ and distinct primes $p$ and $q$ such that $o(x^{t})=p$ and $o(y^{s})=q$.
Now, $x\sim x^{t}\sim y^{s}\sim y$.
\end{proof}

Describing a nilpotent group $G$ with $\textnormal{diam}(\Delta(G))=3$ is straightforward.
Let $G$ be a nilpotent group where every Sylow subgroup is neither cyclic nor generalized quaternion.
Write $p_{1},\dots, p_{t}$, $t\ge 2$, for the prime divisors of $|G|$.
Each Sylow $p_{i}$-subgroup of $G$ has at least two subgroups of order $p_{i}$.
Let $\langle x_{i}\rangle$ and $\langle y_{i}\rangle$ be distinct subgroups of order $p_{i}$ for each $i=1,\dots, t$.
Let $x=\prod_{i=1}^{t}x_{i}$ and $y=\prod_{i=1}^{t}y_{i}$.
Note that $x_{1}$ is a power of $x$ and that $y_{1}$ is a power of $y$.
If $x\sim y$, then $\langle x,y\rangle$ has a unique subgroup of order $p_{1}$.
Hence $\langle x_{1}\rangle=\langle y_{1}\rangle$, a contradiction.
Assume there exists $z\in G^{\#}$ with $x\sim z\sim y$.
For some integer $k$, the element $z^{k}$ has order $p_{j}$ for some $j\in\{1,\dots,t\}$.
Note that $x_{j}$ is a power of $x$ and that $y_{j}$ is a power of $y$.
As $x\sim z$ and $z\sim y$, $\langle x_{j}\rangle=\langle z^{k}\rangle=\langle y_{j}\rangle$, another contradiction.
Hence, $d(x,y)=3$.
Finally, we note that if $G$ is a nilpotent group where at least one, but not all, Sylow subgroup is cyclic or generalized quaternion, then $\textnormal{diam}(\Delta(G))=2$.

\section{Centers and Centralizers}

Let $G$ be a group.
In this section, the graph $\Delta(G)$ is characterized under various assumptions on $\centerof{G}$ and on centralizers of elements in $G$.

The structure of $\Delta(G)$ is easily determined when $\centerof{G}$ does \textit{not} have prime power order.
The following lemma recovers essentially Theorem 5.2 in \cite{bera}.
See also Lemma 9 in \cite{shi} for the punctured power graph version of this result.

\begin{lemma}\label{orderofcenterdivisiblebytwoprimes}
If $G$ is a group and $|\centerof{G}|$ is divisible by at least two distinct primes, then $\Delta(G)$ is connected with $\textnormal{diam}(\Delta(G))\le 4$.
\end{lemma}

\begin{proof}
Let $p$ and $q$ be distinct prime divisors of $|\centerof{G}|$.
Let $a,b\in\centerof{G}$ with $o(a)=p$ and $o(b)=q$.
Note that $a\sim b$ by Lemma \ref{coprimeandcommute}.

Let $x\in G^{\#}$.
If $o(x)$ is divisible by a prime $r\neq p$, then $o(x^{t})=r$ for a suitable integer $t$.
Now, $x\sim x^{t}\sim a$.
If $x$ is a $p$-element, then $x \sim b$, and so $x\sim b\sim a$.
Hence $d(x,a)\le 2$ for every $x\in G^{\#}$, and the result follows.
\end{proof}

The group $G=C_{6}\times S_{3}\times S_{3}$ shows that $\textnormal{diam}(\Delta(G))=4$ is possible; see Figure 1.

\begin{figure*}[h]
	\centering
	\includegraphics[width=0.65\linewidth]{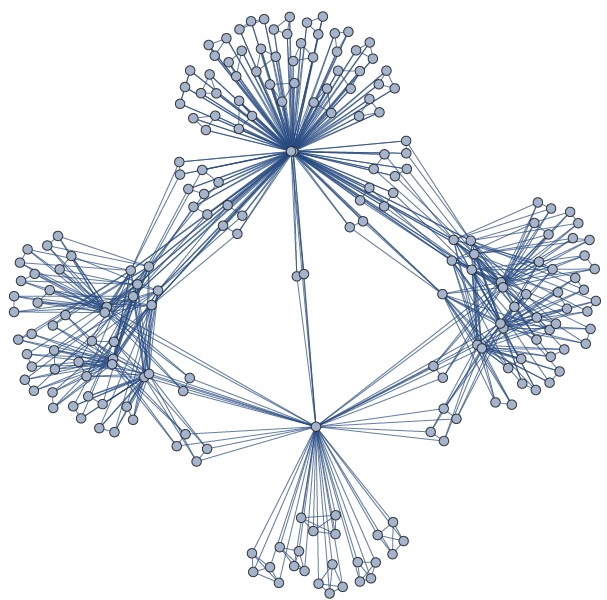}
	\caption{$\Delta(C_{6}\times S_{3}\times S_{3})$}
\end{figure*}

Let $G$ be a group, and assume that $\centerof{G}$ is a $p$-group for some prime $p$.
If $G$ is $p$-group, then $\Delta(G)$ has already been described; see Lemma \ref{subgroupsoforderp}.
So it is no loss to proceed with the assumption $G$ is \textit{not} a group of prime power order.
Let $G$ be a group.
We now introduce two useful conditions on centralizers and we study the consequences of these conditions.
Centralizer conditions for the cyclic graph have been studied briefly; for example, see Theorem 5.3 in \cite{bera}.
Centralizer conditions for the punctured power graph have also been considered; for example, see Lemma 10 in \cite{shi}.

\begin{description}
\item[Condition $(\mathcal{C})$] $\centerof{G}$ is a nontrivial $p$-group for some prime $p$ and $\centralizer{G}{x}$ is \textit{not} a $p$-group for any element $x$ of order $p$.

\item[Condition $(\mathcal{C}_{0})$] $G$ is \textit{not} a group of prime power order and there exists $x\in G$ with prime order, say $p$, such that $\centralizer{G}{x}$ is a $p$-group.
\end{description}

Imposing the condition that $G$ satisfies $(\mathcal{C})$ ensures connectedness of the graph $\Delta(G)$.

\begin{lemma}\label{centralizersnotpgroups}
If $G$ satisfies $(\mathcal{C})$, then $\Delta(G)$ is connected with $\textnormal{diam}(\Delta(G))\le 6$.
\end{lemma}

\begin{proof}
As $G$ satisfies $(\mathcal{C})$, the subgroup $\centerof{G}$ is a nontrivial $p$-group for some prime $p$.
Fix $z\in\centerof{G}^{\#}$.
Let $x\in G^{\#}$.
For suitable integer $t$, $o(x^{t})=q$, where $q$ is a prime.
If $q\neq p$, then $x\sim x^{t}\sim z$.
Assume $q=p$.
Since $\centralizer{G}{x^{t}}$ is not a $p$-group, there exists some $y\in\centralizer{G}{x^{t}}$ with $o(y)=r$, where $r$ is a prime different from $p$.
Now, $x\sim x^{t}\sim y\sim z$.
Hence $d(x,z)\le 3$ for arbitrary $x\in G^{\#}$.
Finally, if $x,y\in G^{\#}$, then a path between $x$ and $y$ of length at most six can be built by passing through $z$.
The result follows.
\end{proof}

The group $G=((C_{29}\times C_{29})\rtimes(C_{15}\rtimes C_{4}))\times C_{2}$, where the left factor is PrimitiveGroup($29^{2}$, $54$) from the groups library in GAP, shows that $\textnormal{diam}(\Delta(G))= 6$ is possible.

Lemma \ref{centralizersnotpgroups} suggests a possible relationship between the existence of a $p$-group centralizer and the disconnectedness of the cyclic graph; this is indeed the case.

\begin{lemma}
If $G$ satisfies $(\mathcal{C}_{0})$, then $\Delta(G)$ is disconnected.
\end{lemma}

\begin{proof}
Our hypothesis delivers an element $x\in G$ with $o(x)=p$, a prime, such that $\centralizer{G}{x}$ is a $p$-group.
Let $\Upsilon$ be the connected component of $\Delta(\centralizer{G}{x})$ containing $x$, and let $\Xi$ be the connected component of $\Delta(G)$ containing $x$.
Clearly $\Upsilon\subseteq\Xi$.

If $\Xi\nsubseteq\Upsilon$, then there exist elements $g\in\Xi\setminus\Upsilon$ and $y\in\Upsilon$ such that $g\sim y$.
Note that $x\sim y$ by applying Lemma \ref{orderpelements} in $\centralizer{G}{x}$.
But now $x=y^{t}$ some integer $t$.
As $g\in\centralizer{G}{y}\le\centralizer{G}{y^{t}}$, it folllows that $g\in\Upsilon$, a contradiction.
Hence $\Xi\subseteq\Upsilon$.
Conclusion: $\Xi=\Upsilon$.

Every element in $\Xi$ is a $p$-element.
Since $G$ is not a $p$-group, the set $G^{\#}\setminus\Xi$ is nonempty, establishing the disconnectedness of $\Delta(G)$.
\end{proof}

The approach thus far has been to place hypotheses on a group $G$ and study $\Delta(G)$.
But with the previous results in place, information about a group can be deduced based on a description of its cyclic graph.
The following lemma is a result in this direction; but more importantly, Lemma \ref{graphconnectedbigdiameter} will serve as a crucial reduction result in the proof of Theorem \ref{directproductconnected}.

\begin{lemma}\label{graphconnectedbigdiameter}
If $G$ is a group, $\Delta(G)$ is connected, and $\textnormal{diam}(\Delta(G))>6$, then $\centerof{G}=\id$.
\end{lemma}

\begin{proof}
By Lemma \ref{orderofcenterdivisiblebytwoprimes}, $\centerof{G}$ must be a $p$-group for some prime $p$.
If $G$ is a $p$-group, then Lemma \ref{subgroupsoforderp} yields that $G$ must have a unique subgroup of order $p$, otherwise $\Delta(G)$ would be disconnected.
Hence $G$ is cyclic or generalized quaternion.
But the hypothesis that $\textnormal{diam}(\Delta(G))>6$ rules out both of these scenarios, and so $G$ is \textit{not} a group of prime power order.

Since $\Delta(G)$ is connected, $G$ does not satisfy $(\mathcal{C}_{0})$.
It follows that, for every $x\in G^{\#}$ with prime order, $|\centralizer{G}{x}|$ is divisible by at least two primes.
If $\centerof{G}\neq\id$, then $G$ satisfies $(\mathcal{C})$, and so $\textnormal{diam}(\Delta(G))\leq 6$ by Lemma \ref{centralizersnotpgroups}, a contradiction.
Hence $\centerof{G}=\id$.
\end{proof}

We close this section by summarizing what classes of groups the Conditions $(\mathcal{C})$ and $(\mathcal{C}_{0})$ have covered and by mentioning what classes they have missed.

The graph $\Delta(G)$ for a group $G$ with $|\centerof{G}|$ divisible by at least two distinct prime was handled in Lemma \ref{orderofcenterdivisiblebytwoprimes}.

Assume that $\centerof{G}$ is a nontrivial $p$-group for some prime $p$.
If the whole group $G$ is a $p$-group, then appeal to Lemma \ref{subgroupsoforderp}.
So suppose that $G$ is \textit{not} a $p$-group.
If $G$ fails to satisfy $(\mathcal{C})$, then there exists an element $x$ of order $p$ with a $p$-group centralizer; hence $G$ satisfies $(\mathcal{C}_{0})$ and $\Delta(G)$ is disconnected.

Next, assume that $\centerof{G}=\id$.
Since the case $G=\id$ can be dismissed, the group $G$ is not a group of prime power order.
If $G$ fails to satisfy $(\mathcal{C}_{0})$, then $\centralizer{G}{x}$ is not a group of prime power order whenever the element $x\in G$ has prime order.

\section{Direct Products}

The cyclic graph of a direct product of two groups $G$ and $H$, our primary focus, is considered in this section.
As a rule, $G$ and $H$ are \textit{nontrivial} groups.

\begin{lemma}\label{coprimegroups}
If $G$ and $H$ are groups with coprime orders, then $\Delta(G\times H)$ is connected with $\textnormal{diam}(\Delta(G\times H))\le 3$.
\end{lemma}

\begin{proof}
Let $(a,b),(c,d)\in(G\times H)^{\#}$.
If $a,c\in G^{\#}$ and $b,d\in H^{\#}$, then $(a,b)\sim(a,1)\sim(1,d)\sim(c,d)$.
If $a=c=1$, then fix $g\in G$ and observe that $(1,b)\sim(g,1)\sim(1,d)$.
If $a=1$ and $c\neq 1$, then $(1,b)\sim(c,1)\sim(c,d)$.
By symmetry, all cases have been exhausted.
\end{proof}

Observe that Lemma \ref{nilpotentatleasttwoprimes} is a corollary of the previous result.
So the family of examples constructed immediately following Lemma \ref{nilpotentatleasttwoprimes} constitutes a collection of direct products satisfying the conditions of Lemma \ref{coprimegroups} and having the largest possible diameter.

Let $G$ and $H$ be nontrivial groups.
The connectedness of $\Delta(G\times H)$ comes with a diameter bound.

\begin{theorem}\label{directproductconnected}
If $G$ and $H$ are nontrivial groups and the graph $\Delta(G\times H)$ is connected, then $\textnormal{diam}(\Delta(G\times H))\le 7$.
\end{theorem}

\begin{proof}
Let $X=G\times H$.
If $\textnormal{diam}(\Delta(X))\le 6$, then the conclusion is obtained.
Assume that $\textnormal{diam}(\Delta(X))> 6$.
The hypotheses of Lemma \ref{graphconnectedbigdiameter} are therefore fulfilled, and so $\centerof{X}=\id$.
Hence $X$ is not a group of prime power order.
In particular, as $\centerof{X}=\centerof{G}\times \centerof{H}$, neither $G$ nor $H$ is a group of prime power order.
Finally, note that the connectedness of $\Delta(X)$ implies that $X$ does not satisfy $(\mathcal{C}_{0})$.

Let $x,y\in X^{\#}$.
For suitable integers $s$ and $t$, $o(x^{s})=p$ and $o(y^{t})=q$, where $p$ and $q$ are primes.
Write $x^{s}=(a,b)$ and $y^{t}=(c,d)$.
Since $X$ does not satisfy $(\mathcal{C}_{0})$, $\centralizer{X}{x^{s}}$ is not a $p$-group.
Let $r\neq p$ be a prime divisor of $|\centralizer{X}{x^{s}}|$.
Because $\centralizer{X}{x^{s}}=\centralizer{G}{a}\times\centralizer{H}{b}$, the prime $r$ divides $|\centralizer{G}{a}|$ or $|\centralizer{H}{b}|$. 
If $r$ divides $|\centralizer{G}{a}|$, then there exists $a'\in\centralizer{G}{a}$ with $o(a')=r$.
Let $x_{0}=(a',1)$, and observe that $x^{s}\sim x_{0}\sim (a,1)$ and $x^{s}\sim x_{0}\sim (1,b)$.
(Notation is being slightly abused as one of the points $(a,1)$ and $(1,b)$ may be the identity.)
If $r$ divides $|\centralizer{H}{b}|$, then there exists $b'\in\centralizer{H}{b}$ with $o(b')=r$.
In this case, take $x_{0}=(1,b')$, and note that $x^{s}\sim x_{0}\sim (a,1)$ and $x^{s}\sim x_{0}\sim (1,b)$.
Conclusion: $d(x^{s},e)\le 2$ for $e\in\{(a,1),(1,b)\}$.
As $x^{s}\neq 1$, the set $\{(a,1),(1,b)\}$ indeed contains a vertex in $\Delta(X)$.
Applying the same argument to $y^{t}$, there exists $y_{0}\in X$ such that $y^{t}\sim y_{0}\sim f$ for $f\in\{(c,1),(1,d)\}$.

Assume that $p\neq q$.
If $a\neq 1$ and $c=1$, then $$x\sim x^{s}\sim x_{0}\sim (a,1)\sim y^{t}\sim y.$$
If $a=1$ and $c\neq 1$, then $$x\sim x^{s}\sim (c,1)\sim y_{0}\sim y^{t}\sim y.$$

Assume now that $a=1$ and $c=1$.
Because $G$ is not a $p$-group, there exists $g\in G$ with $o(g)=r$, where $r$ is a prime different from $p$.
Note that $x\sim x^{s}\sim (g,1)$.
If $r\neq q$, then $$x\sim x^{s}\sim (g,1)\sim y^{t}\sim y.$$
If $r=q$, then $(g,d)$ has order $q$.
Hence $\centralizer{X}{(g,d)}$ is not a $q$-group.
Let $x_{1}\in\centralizer{X}{(g,d)}$ with prime order different from $q$.
Now, $$x\sim x^{s}\sim (g,1)\sim x_{1}\sim y^{t}\sim y.$$

Finally, assume that $a\neq 1$ and $c\neq 1$.
If $b\neq 1$, then $$x\sim x^{s}\sim x_{0}\sim (1,b)\sim (c,1)\sim y_{0}\sim y^{t}\sim y.$$
If $d\neq 1$, then $$x\sim x^{s}\sim x_{0}\sim (a,1)\sim (1,d)\sim y_{0}\sim y^{t}\sim y.$$
Assume $b=1$ and $d=1$.
Now, proceed as in the previous paragraph to see that $d(x,y)\le 5$.
All cases have been exhausted under the assumption that $p\neq q$.

Assume that $p=q$.
Again, there are four cases to consider: (i) $a\neq 1, c=1$, (ii) $a=1,c\neq 1$, (iii) $a=1,c=1$, and (iv) $a\neq 1$, $c\neq 1$.

Assume that $a\neq 1$ and $c=1$.
The element $(a,d)$ has order $p$, and so $\centralizer{X}{(a,d)}$ is not a $p$-group.
Hence, there exists $x_{1}\in\centralizer{X}{(a,d)}$ with $o(x_{1})=r$, where $r\neq p$ is a prime. 
Now, $$x\sim x^{s}\sim x_{0}\sim (a,1)\sim x_{1}\sim y^{t}\sim y.$$
Cases (i) and, by symmetry, (ii) are now resolved.

Assume that $a=1$ and $c=1$.
Since $G$ is not a $p$-group, there exists $g\in G$ with $o(g)=r$, where $r$ is a prime different from $p$.
Now, $$x\sim x^{s}\sim (g,1)\sim y^{t}\sim y,$$ resolving Case (iii).

Assume that $a\neq 1$, $c\neq 1$.
Recall that the element $y_{0}$ has order $r$, where $r$ is a prime number not equal to $p$.
Moreover, by construction, $y_{0}$ has the form $(c',1)$ or $(1,d')$, where $c'\in\centralizer{G}{c}$, $d'\in\centralizer{H}{d}$.
If $b=1$, then this case is handled by a previous argument.
So assume $b\neq 1$.
Now, $x_{0}$ can be chosen so that $x^{s}\sim x_{0}\sim e$ for $e\in\{(a,1),(1,b)\}$.
If $y_{0}$ has the form $(c',1)$, then let $e=(1,b)$.
If $y_{0}$ has the form $(1,d')$, then let $e=(a,1)$.
Hence, $$x\sim x^{s}\sim x_{0}\sim e\sim  y_{0}\sim y^{t}\sim y,$$ which handles Case (iv).

All cases have been exhausted under the assumption that $p=q$.
The proof is complete.
\end{proof}

Let $G$ be SmallGroup(1944,2320) in the GAP Small Groups library, and let $H$ be the Frobenius group $(C_{9}\times C_{9})\rtimes C_{4}$.
Then $\Delta(G\times H)$ has diameter $7$, and so Theorem \ref{directproductconnected} provides the best possible bound.

Let $G$ be a non-cyclic $p$-group and $H$ be a non-cyclic $q$-group for distinct primes $p$ and $q$.
By Lemma \ref{subgroupsoforderp}, the graphs $\Delta(G)$ and $\Delta(H)$ are disconnected.
Lemma \ref{coprimegroups}, however, says that $\Delta(G\times H)$ is connected---and a diameter bound is given by Theorem \ref{directproductconnected}.
It therefore seems as if the connectedness of $\Delta(G\times H)$ for groups $G$ and $H$ may have little or no impact on the graphs $\Delta(G)$ and $\Delta(H)$.
But when the diameter of $\Delta(G\times H)$ is small enough, some information about the cyclic graphs of $G$ and $H$ can be extracted.

\begin{theorem}
If $G$ and $H$ are groups with $\textnormal{diam}(\Delta(G\times H))\le 2$, then $$\textnormal{diam}(\Delta(G))\le 2\qquad\hbox{or}\qquad\textnormal{diam}(\Delta(H))\le 2.$$
\end{theorem}

\begin{proof}
Assume that $3\le\textnormal{diam}(\Delta(G)),\textnormal{diam}(\Delta(H))\le\infty$, and choose $a,c\in G$, $b,d\in H$ such that $d_{\Delta(G)}(a,c),d_{\Delta(H)}(b,d)\ge 3$.
Our hypothesis produces a point $(e,f)\in (G\times H)^{\#}$ satisfying $(a,b)\sim(e,f)\sim(c,d)$.
At least one of the elements $e,f$ is nontrivial, say $e$.
But then $a\sim e\sim c$ in $\Delta(G)$, a contradiction.
\end{proof}

Let $G$ be a group, and consider the following condition:

\begin{description}

\item[Condition $(\mathcal{C}_{1}(p))$] there exists an element in $G$ of prime order $p$ with a $p$-group centralizer.

\end{description}

Note that this condition is Condition $(\mathcal{C}_{0})$ without the hypothesis that $G$ is \textit{not} a $p$-group.
The purpose of isolating this assumption is that it yields a necessary and sufficient condition for the cyclic graph of a direct product of nontrivial groups to be disconnected.

\begin{theorem}\label{disconnectednessc1}
Let $G$ and $H$ be nontrivial groups.
The graph $\Delta(G\times H)$ is disconnected if and only if $G\times H$ satisfies $(\mathcal{C}_{1}(p))$ for some prime $p$.
\end{theorem}

\begin{proof}
Assume that $G\times H$ satisfies $(\mathcal{C}_{1}(p))$ for some prime $p$.
If $G\times H$ is a $p$-group, then clearly $G\times H$ cannot be cyclic or generalized quaternion since both $G$ and $H$ are nontrivial.
Hence $\Delta(G\times H)$ is disconnected.
If $G\times H$ is not a group of prime power order, then $G\times H$ satisfies $(\mathcal{C}_{0})$.
So, again, $\Delta(G\times H)$ is disconnected.

If $G\times H$ does not satisfy $(\mathcal{C}_{1}(p))$ for any prime $p$, then every element of prime order has a centralizer with order divisible by at least two distinct primes.
But notice that this assumption is what produced paths in Theorem \ref{directproductconnected}.
Hence $\Delta(G\times H)$ is connected.
\end{proof}

Let $G$ and $H$ be nontrivial groups.
The disconnectedness of $\Delta(G\times H)$ can also be characterized by a centralizer condition on both direct factors $G$ and $H$.
An easy lemma is needed.

\begin{lemma}\label{c1directfactors}
Let $G$ and $H$ be nontrivial groups, and let $p$ be a prime.
The group $G\times H$ satisfies $(\mathcal{C}_{1}(p))$ if and only if $G$ and $H$ satisfy $(\mathcal{C}_{1}(p))$.
\end{lemma}

\begin{proof}
Let $X=G\times H$.
Assume that $X$ satisfies $(\mathcal{C}_{1}(p))$.
Hence, there exists $x=(g,h)\in X$ with $o(x)=p$, such that $\centralizer{X}{x}$ is a $p$-group.
Note that $o(g),o(h)\in\{1,p\}$ since $1=x^{p}=(g^{p},h^{p})$.
If neither $g$ nor $h$ is the identity element, then $o(g)=o(h)=p$.
As $\centralizer{X}{x}=\centralizer{G}{g}\times\centralizer{H}{h}$, both centralizers $\centralizer{G}{g}$ and $\centralizer{H}{h}$ are $p$-groups, establishing the conclusion.
If $h=1$ (and so $g$ has order $p$), then $H$ is a $p$-group, and so, again, the conclusion follows.
If $g=1$ (and so $h$ has order $p$), then $G$ is a $p$-group.
In all cases, $G$ and $H$ satisfy $(\mathcal{C}_{1}(p))$.

Assume conversely that $G$ and $H$ satisfy $(\mathcal{C}_{1}(p))$.
Our hypothesis therefore delivers an element $g\in G$ with $o(g)=p$ such that $\centralizer{G}{g}$ is $p$-group and an element $h\in H$ with $o(h)=p$ such that $\centralizer{H}{h}$ is $p$-group.
Let $x=(g,h)$, and observe that the existence of $x$ grants $X$ membership into the class of $(\mathcal{C}_{1}(p))$-groups.
\end{proof} 

Theorem \ref{disconnectednessc1} and Lemma \ref{c1directfactors} now yield the following corollary.

\begin{corollary}
Let $G$ and $H$ be nontrivial groups.
The graph $\Delta(G\times H)$ is disconnected if and only if there exists a prime $p$ such that $G$ and $H$ satisfy $(\mathcal{C}_{1}(p))$.
\end{corollary}

\end{document}